\documentclass{amsproc}
\usepackage{eurosym}
\usepackage{amssymb}
\usepackage{amsfonts}

\theoremstyle{plain}

\newtheorem{definition}{Definition}

\newtheorem{proposition}{Proposition}

\newtheorem{theorem}{Theorem}
\numberwithin{equation}{section}

\begin{document}
\title[Approximately CM complexes]{Equivalent condition for approximately
Cohen-Macaulay complexes}
\author{Micha\l\ Laso\'{n}}
\address{Institute of Mathematics of the Polish Academy of Sciences, \'{S}%
niadeckich 8, 00-956 Warszawa, Poland}
\address{Theoretical Computer Science Department, Faculty of Mathematics and
Computer Science, Jagiellonian University, 30-348 Krak\'{o}w, Poland}
\email{michalason@gmail.com}

\begin{abstract}
We give a necessary and sufficient condition for a simplicial complex to be
approximately Cohen-Macaulay. Namely it is approximately Cohen-Macaulay if
and only if the ideal associated to its Alexander dual is componentwise
linear and generated in two consecutive degrees. This completes the result
of Herzog and Hibi who proved that a simplicial complex is
sequentially Cohen-Macaulay if and only if the ideal associated to its
Alexander dual is componentwise linear.
\end{abstract}

\maketitle

\section{Introduction}

In \cite{eare} Eagon and Reiner proved that a simplicial complex is
Cohen-Macaulay if and only if the ideal associated to its Alexander dual has
linear resolution. Later Herzog and Hibi \cite{hehi} generalized it
and proved that a simplicial complex is sequentially Cohen-Macaulay if and
only if the ideal associated to its Alexander dual is componentwise linear.
We use their result to give a similar equivalent condition for a simplicial
complex to be approximately Cohen-Macaulay.

We begin with a brief introduction to the topic. When we say that a
simplicial complex is Cohen-Macaulay, sequentially Cohen-Macaulay, or
approximately Cohen-Macaulay, we always think that its Stanley-Reisner ring
has this property.

\begin{definition}
For a simplicial complex $\Delta $ on the set of vertices $\{1,\dots ,n\}$
and a field $\mathbb{K}$, the \emph{Stanley-Reisner ring} (or \emph{face ring%
}) is the ring $\mathbb{K}[x_{1},\dots ,x_{n}]/I_{\Delta}=\mathbb{K}[\Delta ]$, where 
$I_{\Delta}$ is generated by all squarefree monomials $x_{i_{1}}\cdots x_{i_{l}}$ for
which \newline $\{i_{1},\dots ,i_{l}\}$ is not a face in $\Delta $.
\end{definition}

We recall combinatorial description of Cohen-Macaulay complexes:

\begin{definition}
Let $\sigma $ be a simplex in a simplicial complex $\Delta $. The \emph{link}
of $\sigma $ in $\Delta $, denoted by $lk_{\Delta }\sigma $, is the
simplicial complex $\{\tau \in \Delta :\tau \cap \sigma =\emptyset \;\text{%
and}\;\tau \cup \sigma \in \Delta \}$.
\end{definition}

\begin{theorem}
\label{reisner} \emph{(Reisner \cite{re})} Let $R=\mathbb{K}[\Delta ]$ be
the face ring of $\Delta $. Then the following conditions are equivalent:

\begin{enumerate}
\item $R$ is Cohen-Macaulay ring.

\item $\tilde{H}_{i}(lk_{\Delta }\sigma )=0$ if $i<\dim (lk_{\Delta }\sigma )
$ for all simplices $\sigma \in \Delta $.
\end{enumerate}
\end{theorem}

For some techniques of counting homology we refer the reader to Section 3.2
of \cite{lami}. We need also the following definitions:

\begin{definition}\emph{\cite{go}}
A non Cohen-Macaulay local ring $A$ is called \emph{approximately
Cohen-Macaulay} if there is an element $a$ in the maximal ideal such that $%
A/(a^{n})$ is Cohen-Macaulay ring of dimension $\dim (A)-1$ for all $n>0$.
\end{definition}

\begin{definition}
A ring $A$ of dimension $d$ is called \emph{sequentially Cohen-Macaulay} if
there exists a filtration of ideals of $A$: 
\begin{equation*}
0=D_{0}\subset D_{1}\subset \dots \subset D_{t}=A
\end{equation*}%
such that each $D_{i}/D_{i-1}$ is Cohen-Macaulay and 
\begin{equation*}
0<\dim (D_{1}/D_{0})<\dim (D_{2}/D_{1})<\dots <\dim (D_{t}/D_{t-1})=d.
\end{equation*}
\end{definition}

\begin{definition}
Let $\Delta $ be a simplicial complex on the set of vertices $V$, we define
its \emph{Alexander dual} to be $\Delta ^{\ast }=\{V\setminus \sigma :\sigma
\notin \Delta \}$.
\end{definition}

\begin{definition}
We say that a graded ideal $I\subset A$ is \emph{componentwise linear} if $%
I_{j}$ has linear resolutions for each degree $j$.
\end{definition}

There is a nice description of approximately Cohen-Macaulay rings:

\begin{proposition}
\label{approx}\emph{\cite{cucu}} Let $A$ be a non Cohen-Macaulay local ring
of dimension $d$. Then the following conditions are equivalent:

\begin{enumerate}
\item $A$ is an approximately Cohen-Macaulay ring.

\item $A$ is a sequentially Cohen-Macaulay ring with filtration $%
0=D_{0}\subset D_{1}\subset D_{2}=A$, where $\dim (D_{1})=d-1$.
\end{enumerate}
\end{proposition}

\section{Equivalent condition}

We will make use of the following result of Herzog and Hibi \emph{\cite{hehi}%
}.

\begin{theorem}
\label{seqen}\emph{\cite{hehi}} Let $\Delta $ be a simplicial complex. Then
Stanley-Reisner ring $\mathbb{K}[\Delta ]$ is sequentially Cohen-Macaulay if
and only if $I_{\Delta^{\ast } },$ the ideal associated to its Alexander dual,
is componentwise linear.
\end{theorem}

Our theorem reads as follows.

\begin{theorem}
Let $\Delta $ be a simplicial complex. Then the Stanley-Reisner ring $%
\mathbb{K}[\Delta ]$ is approximately Cohen-Macaulay if and only if $%
I_{\Delta^{\ast } }$, ideal associated to its Alexander dual, is componentwise
linear and generated in two consecutive degrees.
\end{theorem}

\begin{proof}
By Proposition \ref{approx}, $\mathbb{K}[\Delta ]$ is approximately
Cohen-Macaulay if and only if $\mathbb{K}[\Delta ]$ is a sequentially
Cohen-Macaulay ring with filtration 
\begin{equation*}
0=D_{0}\subset D_{1}\subset D_{2}=\mathbb{K}[\Delta ],
\end{equation*}%
where $\dim (D_{1})=d-1$. Due to the Theorem \ref{seqen} of Herzog and Hibi
this is equivalent to componentwise linearity of $I_{\Delta ^{\ast } }$, and
existence of a filtration 
\begin{equation*}
0=D_{0}\subset D_{1}\subset D_{2}=\mathbb{K}[\Delta ],
\end{equation*}%
where $\dim (D_{1})=d-1$. From Appendix of \cite{bjwawe} we get that if such a
filtration exists, then it is unique and coincides with the one given by 
\begin{equation*}
0=M_{0}\subset \dots \subset M_{i-1}=I_{\Delta ,\Delta ^{\langle
j_{i}-1\rangle }}\subset \dots \subset \mathbb{K}[\Delta ],
\end{equation*}%
where $I_{\Delta ,\Delta ^{\langle j_{i}-1\rangle }}$ is the ideal in $%
\mathbb{K}[\Delta ]$ generated by monomials $x_{A}$, with $A\in \Delta
\setminus \Delta ^{\langle j_{i}-1\rangle }$. The simplicial complex $\Delta
^{\langle j_{i}-1\rangle }$ is generated by faces of $\Delta $ of dimension
\ at least $j_{i}-1$, where $j_{1}-1<\dots <j_{s}-1$ are the dimensions of
facets of $\Delta $. We have also that $\dim (\Delta ^{\langle
j_{i}-1\rangle })=j_{i}-1$. Hence the desired filtration exists if and only
if $\Delta $ has facets of dimension $d$ and $d-1$. Ideal $I_{\Delta ^{\ast
}}$ is generated by monomials $x_{A}$ for $A\notin \Delta ^{\ast }$, that
is, for $A=V\setminus \sigma $, where $\sigma \in \Delta $. We have to take
all $x_{A}$ corresponding to facets and they all already generate ideal.
Hence the ideal $I_{\Delta ^{\ast }}$ is generated in two consecutive
degrees $v-d$ and $v-(d-1)$, where $\left\vert V\right\vert =v$. Since each
step of our reasoning was an equivalence, the contrary also holds.
\end{proof}

\section*{Acknowledgements}

I would like to thank Ralf Fr\"{o}berg for many inspiring conversations and
introduction to this subject. I would also like to thank Jarek Grytczuk for
help in preparation of this manuscript. Also I acknowledge a support from
Polish Ministry of Science and Higher Education grant MNiSW N N201 413139.

\end{document}